\documentclass[reqno,12pt,letterpaper]{amsart}
\usepackage{amsmath,amssymb,amsthm,graphicx,mathrsfs,url,bbm}
\usepackage[usenames,dvipsnames]{color}
\usepackage[colorlinks=true,linkcolor=Red,citecolor=Green]{hyperref}
\usepackage{amsxtra}

\usepackage[
    backend=biber,
    style=alphabetic,
    giveninits=true,
    maxnames=10
]{biblatex}
\DeclareFieldFormat{pages}{#1}
\renewbibmacro{in:}{%
  \ifentrytype{article}
    {}
    {\bibstring{in}%
     \printunit{\intitlepunct}}}
\DeclareFieldFormat
  [article,inbook,incollection,inproceedings,patent,thesis,unpublished]
  {title}{\mkbibemph{#1}}
\DeclareFieldFormat{journaltitle}{#1\isdot}
\DeclareFieldFormat[article]{volume}{\mkbibbold{#1}}
\DeclareFieldFormat[article]{number}{\bibstring{number}\addnbspace #1}

\renewbibmacro*{journal+issuetitle}{%
  \usebibmacro{journal}%
  \setunit*{\addspace}%
  \iffieldundef{series}
    {}
    {\newunit
     \printfield{series}%
     \setunit{\addspace}}%
  \printfield{volume}%
  \setunit{\addspace}%
  \usebibmacro{issue+date}%
  \setunit{\addcomma\space}%
  \printfield{number}%
  \setunit{\addcolon\space}%
  \usebibmacro{issue}%
  \setunit{\addcomma\space}%
  \printfield{eid}
  \newunit}

\def\?[#1]{\textbf{[#1]}\marginpar{\Large{\textbf{??}}}}

\let\epsilon=\varepsilon 

\newcommand{\RR}{{\mathbb R}}

\newcommand{\NN}{{\mathbb N}}

\newcommand{\bm}{{\rm b}}
\newcommand{\Hb}{H_{\rm b}}

\newtheorem{thm}{Theorem}
\newtheorem{prop}{Proposition}
\newtheorem{defi}[prop]{Definition}

\newtheorem{lem}[prop]{Lemma}
\newtheorem{corr}[prop]{Corollary}
\newtheorem{rem}{Remark}

\numberwithin{equation}{section}
\newtheorem{que}{Question}

\DeclareMathOperator{\Div}{div}

\DeclareMathOperator{\supp}{supp}

\DeclareMathOperator{\tr}{tr}

\newcommand{\rd}{\partial}
\newcommand{\nb}{\nabla}

\makeatletter
\newcommand{\nrm}{\@ifstar{\nrmb}{\nrmi}}
\newcommand{\nrmi}[1]{\Vert{#1}\Vert}
\newcommand{\nrmb}[1]{\left\Vert{#1}\right\Vert}
\newcommand{\abs}{\@ifstar{\absb}{\absi}}
\newcommand{\absi}[1]{\vert{#1}\vert}
\newcommand{\absb}[1]{\left\vert{#1}\right\vert}
\newcommand{\brk}{\@ifstar{\brkb}{\brki}}
\newcommand{\brki}[1]{\langle{#1}\rangle}
\newcommand{\brkb}[1]{\left\langle{#1}\right\rangle}
\newcommand{\set}{\@ifstar{\setb}{\seti}}
\newcommand{\seti}[1]{\{#1\}}
\newcommand{\setb}[1]{\left\{ #1\right\}}
\makeatother

\title{Localized initial data for Einstein equations}
\author{Yuchen Mao}
\email{yuchen\_mao@berkeley.edu}
\address{Department of Mathematics, Evans Hall, University of California,
Berkeley, CA 94720, USA}
\author{Zhongkai Tao}
\email{ztao@math.berkeley.edu}
\address{Department of Mathematics, Evans Hall, University of California,
Berkeley, CA 94720, USA}

\begin{document}
\begin{abstract}
We apply a new method with explicit solution operators to construct asymptotically flat initial data sets of the vacuum Einstein equation with new localization properties. Applications include an improvement of the decay rate in Carlotto--Schoen \cite{glue2} to $\mathcal{O}(|x|^{-(d-2)})$ and a construction of nontrivial asymptotically flat initial data supported in a degenerate sector $\{(x',x_d)\in\mathbb{R}^d:|x'|\leq x_d^\alpha\}$ for $\frac{3}{d+1}<\alpha<1$.
\end{abstract}

\maketitle


\section{Introduction}
In this note we provide a simple way to construct asymptotically flat initial data of the Einstein equation with new localization properties. The vacuum Einstein equation reads
\begin{align*}
    Ric_g=0
\end{align*}
where $g$ is a Lorentzian metric and $Ric_g$ is the Ricci curvature. When we restrict to a spacelike hypersurface, we get the Einstein constraint equation 
\begin{align}\label{cons}
\left\{\begin{array}{ll}
    R_g+({\rm tr}_g k)^2-|k|_g^2=0\\
    \Div_g (k-({\rm tr}_g k)g)=0 
\end{array}\right.
\end{align}
It is a system of nonlinear underdetermined PDEs for initial data $(g,k)$ on a spacelike hypersurface. When $k=0$, it specializes to a problem in Riemannian geometry, namely vanishing of scalar curvature. In particular, we are interested in the following question.
\begin{que}
What localization of asymptotically flat solutions to the Einstein constraint equation \eqref{cons} is possible?
\end{que}

This question has surprisingly nontrivial answers. The famous positive mass theorem \cite{schoen1979proof,schoen1981proof,witten1981new} says localization to a compact set is impossible. On the positive direction, Carlotto--Schoen \cite{glue2} gives a gluing construction which gives a localized solution inside a cone. Aretakis--Czimek--Rodnianski \cite{aretakis2021characteristic3,aretakis2021characteristic1,aretakis2021characteristic2} gives an alternative proof of the gluing construction based on their characteristic gluing.

The construction in \cite{glue2} loses the decay rate a little, so they cannot get the ideal $\mathcal{O}(|x|^{2-d})$ decay. Carlotto \cite[Open Problem 3.18]{review} conjectured that we can get this optimal decay. Aretakis--Czimek--Rodnianski \cite{aretakis2021characteristic1} gives an affirmative answer. Here we give an alternative proof. Moreover, we can construct solutions with better localization properties, namely in a degenerate sector $\{(x',x_d)\in\mathbb{R}^d:|x'|\leq x_d^\alpha\}$ for $\frac{3}{d+1}<\alpha<1$. When $\alpha$ is close to $1$, this is still in the range of positive mass theorem (see Carlotto \cite[Appendix B]{review}). We state our main results below.

\subsection*{Main results}
As the first application of our method, we give a simple proof of \cite[Open Problem 3.18]{review}.

\begin{thm}\label{thm1}
Let $d\geq 3, s>\frac{d}{2}, -\frac{d}{2}<\delta<\frac{d}{2}-2$. For $\omega \in \mathbb{S}^{d-1}$ and $0<\theta<\pi/2$, let 
\begin{align*}
    \Omega=\Omega_{\omega,\theta}:=\{x\in \mathbb{R}^d: \angle (x,\omega)\leq\theta\}
\end{align*}
be the cone in $\mathbb{R}^d$ with center vector $\omega$ and angle $\theta$. Then there exists a nontrivial asymptotically flat solution $(g,k)$ of equation \eqref{cons} on $\mathbb{R}^d$ supported in the cone $\Omega$, in the sense that
\begin{align*}
    (g^{ij}-\delta^{ij},k^{ij})\in \Hb^{s,\delta}(\mathbb{R}^d)\times \Hb^{s-1,\delta+1}(\mathbb{R}^d),\quad \supp (g^{ij}-\delta^{ij},k^{ij})\subset \Omega.
\end{align*}
 Moreover, the set of such solutions form a smooth submanifold in a neighbourhood of  $0\in \Hb^{s,\delta}(\Omega)\times \Hb^{s-1,\delta+1}(\Omega)$. 

Moreover, we can make $(g,k)\in C^\infty(\RR^d)$ and the decay rate of the solution can be made
\begin{align}\label{thm1:decay}
    \partial^l(g^{ij}(x)-\delta^{ij})=\mathcal{O}(\langle x\rangle^{2-d-l}), \quad \partial^lk^{ij}(x)=\mathcal{O}(\langle x\rangle^{1-d-l}), \quad l\leq s-d-2.
\end{align}
\end{thm}
We recall definitions and standard estimates for the $\bm$-Sobolev space $\Hb^{s,\delta}$ in \S \ref{weight}.

We can also prove a similar gluing result as in \cite{glue2} following the same strategy. 
\begin{thm}\label{thm1’}
Let $d\geq 3, s>\frac{d}{2}, -\frac{d}{2}<\delta<\frac{d}{2}-2$. For $y\in\RR^d,\omega \in \mathbb{S}^{d-1}$ and $0<\theta<\pi/2$, let 
\begin{align*}
    \Omega=\Omega_{y,\omega,\theta}:=\{x\in \mathbb{R}^d: \angle (x-y, \omega)\leq\theta\}
\end{align*}
be the cone in $\mathbb{R}^d$ with center at $y$, center vector $\omega$ and angle $\theta$. Let $0<\theta_0<\theta$. Suppose there is a solution $(g_0^{ij}-\delta^{ij},k_0^{ij})\in \Hb^{s,\delta}(\RR^d)\times \Hb^{s-1,\delta+1}(\RR^d)$ solving the equation \eqref{cons} inside $\Omega$, then for $|y|\gg 1$, there exists an asymptotically flat solution $(g,k)$ of equation \eqref{cons} on $\mathbb{R}^d$ such that
\begin{align*}
    (g^{ij}-\delta^{ij},k^{ij})\in \Hb^{s,\delta}(\mathbb{R}^d)\times \Hb^{s-1,\delta+1}(\mathbb{R}^d),\quad (g,k)=
    \left\{\begin{array}{ll}
    (g_0,k_0),  &\Omega_{y,\omega,\theta_0}\setminus B_1(y),   \\
    (\delta,0),     & \mathbb{R}^n\setminus(\Omega_{y,\omega,\theta}\cup B_1(y)).
    \end{array}\right.
\end{align*}
Suppose $(g_0,k_0)$ has decay rate in \eqref{thm1:decay}, then $(g,k)$ also has decay rate in \eqref{thm1:decay}. If $(g_0,k_0)\in C^\infty$, then we also have $(g,k)\in C^\infty$.
\end{thm}

Another natural conjecture that Carlotto made in \cite[Open Problem 3.14]{review} is that whether we can construct solutions localized in a smaller region as long as we do not violate the constraint of the positive mass theorem. We give a partial answer for the case of a degenerate sector.
\begin{thm}\label{thm2}
Let $d\geq 3,\frac{3}{d+1}<\alpha<1$, consider the degenerate sector 
\begin{align*}
    \Omega=\{(x',x_d)\in\mathbb{R}^d: |x'|\leq x_d^\alpha\}.
\end{align*}
If $s\in\NN, , s>\frac{d}{2}+2,\frac{3-(d+3)\alpha}{2}<\delta<\frac{\alpha(d-1)-3}{2}$, then there exists a smooth nontrivial asymptotically flat solution $(g,k)$ of \eqref{cons} supported in $\Omega$, in the sense that
\begin{align*}
    (g^{ij}-\delta^{ij},k^{ij})\in H^{s,\delta}_\alpha(\mathbb{R}^d)\times H^{s-1,\delta+\alpha}_\alpha(\RR^d)\cap C^\infty(\RR^d),\quad \supp (g^{ij}-\delta^{ij},k^{ij})\subset \Omega.
\end{align*}
The decay rate of the solution is given by
\begin{align}\label{eq:deg_decay1}
    \partial_{x'}^{\beta'}\partial_{x_d}^{\beta_d}(g^{ij}-\delta^{ij})=\mathcal{O}(\langle x \rangle^{1-\alpha(d-1)-|\beta'|\alpha-\beta_d}), \quad |\beta|\leq s-d-2
\end{align}
and
\begin{align}\label{eq:deg_decay2}
 \partial_{x'}^{\beta'}\partial_{x_d}^{\beta_d}(k^{ij})=\mathcal{O}(\langle x \rangle^{1-\alpha d-|\beta'|\alpha-\beta_d}), \quad |\beta|\leq s-d-2.
\end{align}
\end{thm}
The anisotropic Sobolev space $H^{s,\delta}_\alpha$ captures the anisotropic behavior \eqref{eq:deg_decay1}\eqref{eq:deg_decay2} and is defined in section \ref{sec:Omega}. Note there is a natural constraint of the range of $\alpha$ we can get. When $\alpha=0$, the localization is impossible due to positive mass theorem. When $\alpha=1$, this reduces to Theorem \ref{thm1}. We restricted to the $k=0$ in an earlier version of this paper but Philip Isett pointed to us the paper \cite{reshetnyak1970estimates} by Reshetnyak which can be used to construct nice solution operators for the symmetric divergence equation. We will give a finer description of this method in a joint paper with Sung-Jin Oh and Philip Isett \cite{div_eq2023}.

The gluing technique in studying \eqref{cons} appeared much earlier in Corvino \cite{glue0} and Corvino--Schoen \cite{glue1}, and was generalized by Chru{\'s}ciel--Delay \cite{chrusciel2003mapping}. It has been developed to give a version of our Theorem \ref{thm1} in Carlotto--Schoen \cite{glue2}. Aretakis--Czimek--Rodnianski \cite{aretakis2021characteristic3,aretakis2021characteristic1,aretakis2021characteristic2} introduced and studied the characteristic gluing problem. Chru\'sciel \cite{chrusciel2019anti} and Carlotto \cite{review} give reviews of current situation and some open problems. 

\subsection*{Main ideas}
Our construction is different from that of \cite{glue2}.
The key to our construction is a solution operator of the linearized equation with good support properties following Oh--Tataru \cite[Section 4]{yangmills}.
The linearized equation at the trivial metric $\delta_{ij}$ (under a change of variables) is
\begin{align}\label{lin}
    \left\{\begin{array}{ll}
    \partial_i\partial_j h^{ij}=0\\
    \partial_i \pi^{ij}=0 
\end{array}\right.
\end{align}
The basic idea behind the proofs of Theorems \ref{thm1} and \ref{thm1’} is to construct a fundamental solution of \eqref{lin} generalizing the fundamental solution for $\partial_j v^j=0$ in \cite{yangmills}. Then we use Picard iteration in appropriate Sobolev spaces to get the solution of the nonlinear equation \eqref{cons}. In the case of the degenerate sector (Theorem \ref{thm2}), we develop a new fundamental solution and introduce anisotropic Sobolev spaces adapted to the degenerate sector. The sharp decay rate is obtained by representing the solution as the solution operator applied to nonlinearity.

\subsection*{Organization of the paper}
In section \ref{2} we give the construction of the solution operator adapted to a cone. In section \ref{weight} we recall several estimates for the $\bm$-Sobolev spaces. In section \ref{4} we prove Theorem \ref{thm1} and \ref{thm1’} using our solution operator. In section \ref{sec:Omega} we adapt our method to the degenerate sector to give a proof of Theorem \ref{thm2}.

\subsection*{Acknowledgement}
We would like to thank Sung-Jin Oh for suggesting the idea of this simple solution operator and for his numerous help throughout this program. We would like to thank Philip Isett for pointing to us the paper \cite{reshetnyak1970estimates}.  We would also like to thank Peter Hintz and Ethan Sussman for helpful discussions about the $\bm$-Sobolev spaces.

\section{Construction of the solution operator for the linearized equation}\label{2}
The crux of our argument is an explicit solution operator. In this section we will show how to construct a solution operator $S:C_0^\infty(\mathbb{R}^d)\to C^\infty(\mathbb{R}^d)$ ($d\geq 3$) such that
\begin{align*}
    \supp f\subset \text{a cone}\Longrightarrow \supp Sf\subset \text{a cone}.
\end{align*}

Unlike in Corvino \cite{glue0}, $S$ does not have cokernel (on appropriate weighted Sobolev spaces) since the support is noncompact. The integration kernel of $S$ will have an appropriate decay property.

\subsection{Linearized problem}

We begin by reformulating \eqref{cons} in a sufficiently flat region. We introduce new variables $(h, \pi)$, defined as follows:
\begin{align}
(h_{ij}, \pi_{ij}) &= (g_{ij} - \delta_{ij} - \delta_{ij} \tr_{\delta} (g - \delta), k_{ij} - \delta_{ij} \tr_{\delta} k) \label{eq:hpi-def}
\end{align}
Observe that the transformation is obviously invertible with the formulae
\begin{align}
(g_{ij}, k_{ij}) &= (\delta_{ij} + h_{ij} - \frac{1}{d-1} \delta_{ij} \tr_{\delta} h, \pi_{ij} - \frac{1}{d-1} \delta_{ij} \tr_{\delta} \pi). \label{eq:hpi2gk}
\end{align}
With respect to the new variables, the left-hand sides of \eqref{cons} may be written as
\begin{align} 
	R[g] &= \rd_{i} \rd_{j} h^{ij} - M_{h}^{(2)}(h, \rd^{2} h) - M_{h}^{(1)}(\rd h, \rd h), \\
	(\tr_{g} k)^{2} - \abs{k}_{g}^{2} &= - M_{h}^{(0)}(\pi, \pi), \\
	g^{jj'} (g^{ii'}\nb_{g; i} k_{i'j'} - \rd_{j'} \tr_{g} k )&= \rd_{i} \pi^{ij} - N_{h}^{(1) j}(h, \rd \pi) - N_{h}^{(0) j}(\rd h, \pi),
\end{align}
where each of $M^{(n)}_{h}(u, v)$ or $N^{(n) j}_{h}(u, v)$ is a linear combination of contraction of $u$ and $v$ with a smooth tensor field (of the appropriate order) on $\RR^{d}$ that depends only on $h$. 

In conclusion, \eqref{cons} takes the form
\begin{align} 
	\rd_{i} \rd_{j} h^{ij} &= M_{h}^{(2)}(h, \rd^{2} h) + M_{h}^{(1)}(\rd h, \rd h) + M_{h}^{(0)}(\pi, \pi), \label{eq:constraint-h} \\
	\rd_{i} \pi^{ij} &= N_{h}^{(1) j}(h, \rd \pi) + N_{h}^{(0) j}(\rd h, \pi). \label{eq:constraint-pi}
\end{align}
The right hand side is viewed as the nonlinearity and its precise form will not matter in this paper. In the following, we study how to solve the linearized equations given by the left hand side of \eqref{eq:constraint-h}\eqref{eq:constraint-pi}, which we call double divergence equation and symmetric divergence equation, respectively.
\subsection{Solution operator for the divergence equations}
The construction of the following solution operators is the main point for Theorem \ref{thm1} and \ref{thm1’}.
\begin{thm}\label{thm:sol}
For any $w\in \mathbb{S}^{d-1}$ and $\chi\in C^\infty(\mathbb{S}^{d-1})$ with $\int_{\mathbb{S}^{d-1}}\chi=1$ and the cone $\overline{\{x\in\mathbb{R}^d: \frac{x}{|x|}\in\supp\chi\}}$ is convex, then there exists $K_\chi(x),L_{\chi,w}\in \mathcal{D}'(\mathbb{R}^d)$ such that
\begin{align*}
    \left\{\begin{array}{ll}
    \partial_i\partial_j K_\chi^{ij}=\delta\\
    \partial_iL_{\chi,w}^{ij}=\delta w^j.
\end{array}\right.
\end{align*}
Moreover, they satisfy the following properties
\begin{itemize}
    \item $K_\chi$ and $L_{\chi,\omega}$ are symmetric $2$-tensors;
    \item The support of $K_\chi,L_{\chi,\omega}$ lie inside the cone
    $\overline{\{x\in\mathbb{R}^d: \frac{x}{|x|}\in\supp\chi\}}$;
    \item $K_\chi$ is homogeneous of degree $2-d$, $L_{\chi,\omega}$ is homogeneous of degree $1-d$.
    \item $K_\chi,L_{\chi,\omega}$ are smooth in $\mathbb{R}^d\setminus\{0\}$.
\end{itemize}
\end{thm}
\begin{proof}
\noindent\textbf{Step 1:}
We first consider a the case of the divergence equation
\begin{align*}
    \partial_i h^i=f.
\end{align*}
Let $\omega\in \mathbb{S}^{d-1}$, $H$ be the Heaviside function, then
\begin{align*}
    T_{\omega}=\omega H(x\cdot \omega)\delta(
    \omega^\perp)
\end{align*}
is a fundamental solution for the divergence equation.
From this we can construct a smoother fundamental solution by averaging in $\omega$. 
Indeed, Let $\chi\in C^\infty(\mathbb{S}^{d-1})$ with $\int_{\mathbb{S}^{d-1}}\chi(\omega)d\omega=1$, then
\begin{align*}
    \langle T_\omega, f\rangle=\int_0^\infty \omega f(t\omega)dt
\end{align*}
and
\begin{align*}
     \left\langle \int_{\mathbb{S}^{d-1}} \chi(\omega)T_\omega d\omega, f\right\rangle&=\int_0^\infty \int_{\mathbb{S}^{d-1}}\omega f(t\omega)\chi(\omega)d\omega dt\\
     &=\int_{\mathbb{R}^d} f(x) \chi\left(\frac{x}{|x|}\right)\frac{x}{|x|^d} dx.
\end{align*}
Thus we have a fundamental solution 
\begin{align}\label{div fund}
    \tilde{K}_{\chi}(x)=\int_{\mathbb{S}^{d-1}} \chi(\omega) T_\omega d\omega=\chi\left(\frac{x}
    {|x|}\right)\frac{x}{|x|^d}
\end{align}
which is homogeneous of degree $1-d$ and smooth outside the origin.

\noindent\textbf{Step 2:}
We can now apply the same idea to construct solution operators of the linearized constraint equations.
\begin{align}
    \partial_i\partial_j h^{ij}=f\label{lin_double}\\
    \partial_i\pi^{ij}=g^j.\label{lin_sym}
\end{align}
For the first equation \eqref{lin_double} we may just apply the previous solution operator twice. The fundamental solution has the integration kernel
\begin{align*}
    K^{ij}_{\chi}=\left(\chi\left(\frac{x}{|x|}\right)\frac{x_i}{|x|^d}\right)*\left(\chi\left(\frac{x}{|x|}\right)\frac{x_j}{|x|^d}\right).
\end{align*}

For the second equation \eqref{lin_sym}, we need to first find singular fundamental solutions as before. For $v,w\in \mathbb{S}^{d-1}$, let
  \begin{align*}
    \pi^{jk}=\partial_l\phi (v^jv^lw^k+v^lv^kw^j-w^lv^kv^j),
\end{align*}
the equation \eqref{lin_sym} becomes
\begin{align*}
    \partial_j\pi^{jk}=\partial_j\partial_l \phi v^j v^l w^k=g^k.
\end{align*}
Then
\begin{align*}
    \langle\phi,f\rangle=\int_0^\infty tf(tv)dt
\end{align*}
gives a fundamental solution $L_{v,w}$ such that $\partial_{j}\pi^{jk}=\delta w^k$ and $\pi^{jk}$ is a symmetric tensor.
Averaging along $v$ as before, we get
\begin{align*}
    \left\langle\int_{\mathbb{S}^{d-1}}\chi(v) L_{v,w}dv, f\right\rangle&=-\int_{\mathbb{S}^{d-1}}\chi(v)\int_0^\infty  t(\partial_l f)(tv)(v^jv^lw^k+v^lv^kw^j-w^lv^kv^j)dtdv\\
    &=-\int_{\mathbb{R}^d}\chi\left(\frac{x}{|x|}\right)(\partial_l f)(x)    \frac{x^jx^l w^k+x^lx^kw^j-w^lx^kx^j}{|x|^d} dx\\
    &=\left\langle  \partial_l\left(\chi\left(\frac{x}{|x|}\right)   \frac{x^jx^l w^k+x^lx^kw^j-w^lx^kx^j}{|x|^d}\right), f\right\rangle.
\end{align*}
So the fundamental solution reads
\begin{align*}
    L_{\chi,w}=\partial_l\left(\chi\left(\frac{x}{|x|}\right)   \frac{x^jx^l w^k+x^lx^kw^j-w^lx^kx^j}{|x|^d}\right).
\end{align*}
All the properties of the solution operators follow directly from the construction.\qedhere

\end{proof}

\section{Estimates on the $\bm$-Sobolev spaces}\label{weight}
In order to capture the decay property of functions and get optimal regularity, we recall basic estimates for $\bm$-Sobolev spaces in this section, and prove our solution operators are bounded on $\bm$-Sobolev spaces.
\subsection{The $\bm$-Sobolev space}
\begin{defi}
For $s\in\NN_0$, the $\bm$-Sobolev space $\Hb^{s}$ is defined by the norm
\begin{align*}
    \|u\|_{\Hb^s}^2:=\sum\limits_{k\leq s}\|\langle x\rangle^k \nabla^k u\|_{L^2(\RR^d)}^2.
\end{align*}
We extend the definition to $s\in\RR$ by duality and (complex) interpolation. We further define for $\delta\in\RR$, $\Hb^{s,\delta}:=\langle x\rangle^{-\delta}\Hb^s$. Moreover, for $\Omega\subset\RR^d$, we denote $\Hb^{s,\delta}(\Omega):=\{u\in \Hb^{s,\delta}:\supp u\subset\overline{\Omega}\}$ the supported distributions.
\end{defi}

The $\bm$-Sobolev space captures the property that the decay rate of a function improves by $\langle x\rangle^{-1}$ after taking derivative. We recall some properties of the $\bm$-Sobolev space.
\begin{prop}\label{prop:b_Sobolev}
\begin{itemize}
    \item Littlewood--Paley decomposition: let
    $\mathcal{D}_0=B(0,1)$ and
    $\mathcal{D}_{j}=\{2^{j-2}<|x|<2^{j+2}\}$ for $j\geq 1$, and $\Phi_0={\rm id}$,
    \begin{align*}
        \Phi_j: (-2,2)\times \mathbb{S}^{d-1}\to \mathcal{D}_j,\quad (w,z)\mapsto 2^{w+j}z,\quad j\geq 1.
    \end{align*}
    Let $\sum \chi_j=1$ be a partition of unity such that $\chi_0\in C_0^\infty(\mathcal{D}_0)$, $\chi\in C_0^\infty((-2,2)\times \mathbb{S}^{d-1})$ and $\chi_j=\Phi_{j*}(\chi)$ for $j\geq 1$. Then 
    \begin{align}\label{eq:L-P decomposition b}
        \|u\|_{\Hb^{s,\delta}}^2\approx \sum\limits_j 2^{2j(\delta+d/2)}\|\Phi_j^*(\chi_j u)\|_{H^s}^2.
    \end{align}
    \item Sobolev embedding: for $s>d/2$, \begin{align}\label{eq:b-Sobolev embedding}
        \|\langle x\rangle^{d/2+\delta}u\|_{L^\infty}\lesssim \|u\|_{\Hb^{s,\delta}}.
    \end{align}
    \item bilinear estimate: the multiplication is bounded on the following spaces
    \begin{align}\label{eq:b-bilinear}
    (u,v)\mapsto uv: \Hb^{s_1,\delta_1}\times \Hb^{s_2,\delta_2}\to \Hb^{s,\delta}
    \end{align}
    for $\delta_1+\delta_2=\delta-d/2$, $s_1+s_2>0$,
    \begin{align*}
    s=\left\{\begin{array}{ll}
        \min(s_1,s_2),     &\max(s_1,s_2)>\frac{d}{2},  \\ 
        s_1+s_2-\frac{d}{2},     &\max(s_1,s_2)\leq \frac{d}{2}. 
        \end{array}\right.
    \end{align*}
\end{itemize}
\end{prop}
\begin{proof}
\begin{itemize}
    \item The Littlewood--Paley decomposition \eqref{eq:L-P decomposition b} is \cite[Lemma 2.3]{hintz2022gluing}. By interpolation and duality, it suffices to check for $s\in\NN_0$, and it is direct to check \eqref{eq:L-P decomposition b} inductively for $s$. Note our convention on $\delta$ is shifted by $d/2$ due to a different choice of density near the boundary.
    \item The Sobolev embedding is essentially \cite[Corollary 2.4]{hintz2022gluing}. It is proved as a corollary of \eqref{eq:L-P decomposition b}:
    \begin{align*}
        \|\langle x\rangle^{d/2+\delta}u\|_{L^\infty}\lesssim \sup_j 2^{j(d/2+\delta)}\|\chi_j u\|_{L^\infty}\lesssim \sup_j 2^{j(d/2+\delta)}\|\Phi_j^*(\chi_j u)\|_{H^s}\lesssim \|u\|_{\Hb^{s,\delta}}.
    \end{align*}
    \item The bilinear estimate \eqref{eq:b-bilinear} is again a corollary of \eqref{eq:L-P decomposition b}:
    \begin{align*}
        \|uv\|_{\Hb^{s,\delta}}^2&\lesssim \sum\limits_j 2^{2j(\delta+d/2)}\|\Phi_j^*(\chi_j u v)\|_{H^s}^2\\
        &\lesssim \sum\limits_j 2^{2j(\delta+d/2)}\|\Phi_j^*(\chi_j u )\|_{H^{s_1}}^2\|\Phi_j^*(\tilde{\chi}_j v )\|_{H^{s_2}}^2\\
        &\lesssim \sum\limits_j 2^{2j(\delta+d/2)}\|\Phi_j^*(\chi_j u )\|_{H^{s_1}}^2\|\Phi_j^*(\tilde{\chi}_j v )\|_{H^{s_2}}^2
        \\&\lesssim \left(\sum\limits_j 2^{2j(\delta_1+d/2)}\|\Phi_j^*(\chi_j u )\|_{H^{s_1}}^2\right)\left(\sum\limits_j 2^{2j(\delta_2+d/2)}\|\Phi_j^*(\tilde{\chi}_j u )\|_{H^{s_2}}^2\right)\\
        &\lesssim \|u\|_{\Hb^{s_1,\delta_1}}^2\|v\|_{\Hb^{s_2,\delta_2}}^2.\qedhere
    \end{align*}
\end{itemize}
\end{proof}

We now show our solution operator is bounded on $\Hb^{s,\delta}$.
\begin{prop}\label{prop:bound_sol}
Let $K\in C^\infty(\RR^d\setminus\{0\})\cap \mathcal{D}'(\RR^d)$ be a homogeneous distribution of order $k-d$, $0<k<d$, then for $-d/2<\delta<d/2-k$, we have
\begin{align}\label{eq:bound_sol}
    u\mapsto K*u :\Hb^{s-k,\delta+k}\to \Hb^{s,\delta}.
\end{align}
\end{prop}
\begin{proof}
We decompose $K(x-y)=K_{\rm diag}+K_{\rm in}+K_{\rm out}$, where 
\begin{align*}
    K_{\rm diag}(x,y)=K(x-y)\varphi\left(\frac{|x|-|y|}{|y|}\right)\chi_{>1}(y)+K(x-y)\chi_{<3}(x)\chi_{<2}(y)
\end{align*}
where $\varphi\in C_0^\infty(-1/2,1/2)$ is a cutoff function such that $\varphi=1$ near $0$, and $\chi_{>1}\in C_0^\infty(\RR^2\setminus \overline{B(0,1)};[0,1])$, $\chi_{<2}\in C_0^\infty(B(0,2);[0,1])$, $\chi_{<3}\in C_0^\infty(B(0,3);[0,1])$ are cutoff functions so that $\chi_{>1} (y)+\chi_{<2}(y)=1$ and $\chi_{<3}(x)=1$ for $|x|\leq 2$; $K_{\rm in}=(K-K_{\rm diag})\mathbbm{1}_{|x|<|y|}$ is incoming part of $K$, and $K_{\rm out}=(K-K_{\rm diag})\mathbbm{1}_{|x|>|y|}$ is the outgoing part of $K$.

The diagonal part is bounded:
\begin{align*}
    \|K_{\rm diag} u\|_{\Hb^{s,\delta}}^2&\approx \sum\limits_j 2^{2j(\delta+d/2)}\left\|\Phi_j^*\chi_j \int K_{\rm diag}(\Phi_j(x),y)u(y)dy\right\|_{H^s}^2\\
    &\approx \sum\limits_j 2^{2j(\delta+d/2)}\left\|\Phi_j^*\chi_j \int K_{\rm diag}(\Phi_j(x),\Phi_j(y))u(\Phi_j(y))|\Phi_j'(y)|dy\right\|_{H^s}^2\\
    &\lesssim \sum\limits_j 2^{2j(\delta+d/2)}2^{2jk}\|\Phi_j^*(\tilde{\chi}_j u)\|_{H^{s-k}}^2\\
    &\approx \|u\|_{\Hb^{s-k,\delta+k}}^2.
\end{align*}
In the second last step we use the fact that $|\Phi_j'(y)|\approx 2^{jd}$ and $K$ is homogeneous of order $k-d$.

The outgoing part is bounded for $\delta<d/2-k$:
\begin{align*}
    \|K_{\rm out} u\|_{\Hb^{s,\delta}}^2&\approx \sum\limits_j 2^{2j(\delta+d/2)}\left\|\Phi_j^*\chi_j \int K_{\rm out}(\Phi_j(x),y)u(y)dy\right\|_{H^s}^2\\
    &\approx \sum\limits_j 2^{2j(\delta+d/2)}\left\|\sum\limits_{j'<j}\Phi_j^*\chi_j \int K_{\rm out}(\Phi_j(x),\Phi_{j'}(y))u(\Phi_{j'}(y))|\Phi_{j'}'(y)|dy\right\|_{H^s}^2\\
    &\lesssim \sum\limits_j 2^{2j(\delta+d/2)}\left(\sum\limits_{j'<j}  2^{j(k-d)}2^{j'd}\|\Phi_{j'}^*(\tilde{\chi}_{j'} u)\|_{H^{-N}}\right)^2\\
    &\lesssim\sum\limits_j 2^{2j(\delta+d/2)}\left(\sum\limits_{j'<j}  2^{j(k-d)}2^{j'd}2^{\beta(j-j')}\right)\left(\sum\limits_{j'<j}  2^{j(k-d)}2^{j'd}2^{\beta(j'-j)}\|\Phi_{j'}^*(\tilde{\chi}_{j'} u)\|_{H^{-N}}^2\right).
\end{align*}
Since $\delta<d/2-k$, we may choose $\beta\in\RR$ such that $d-\beta>0$ and $ 2\delta+2k-\beta<0$,
so
\begin{align*}
    \sum\limits_{j'<j}2^{j(k-d)}2^{j'd}2^{\beta(j-j')}\lesssim 2^{jk},\quad
    \sum\limits_{j'<j}2^{2j(\delta+d/2)}2^{jk}2^{j(k-d)}2^{j'd}2^{\beta(j'-j)}\lesssim 2^{2j'(\delta+d/2+k)}.
\end{align*}
Thus we conclude
\begin{align*}
    \|K_{\rm  out}u\|_{\Hb^{s,\delta}}^2\lesssim \sum\limits_{j'}2^{2j'(\delta+d/2+k)}\|\Phi_{j'}^*(\tilde{\chi}_{j'} u)\|_{H^{-N}}^2\approx \|u\|_{\Hb^{-N,\delta+k}}^2.
\end{align*}
We could similarly conclude the incoming part is bounded:
\begin{align*}
    \|K_{\rm in}u\|_{\Hb^{s,\delta}}\lesssim \|u\|_{\Hb^{-N,\delta+k}}
\end{align*}
as long as $\delta>-d/2$.\qedhere

\end{proof}

\begin{corr}
Let $\Omega\subset \RR^d$ be a closed subset. In Proposition \ref{prop:bound_sol}, if we assume $\supp(K*u)\subset \Omega$ for any $u\in C_c^\infty(\Omega)$, and $K$ is outgoing, i.e. $K_{\rm in}(x,y)=0$ for sufficiently large $(x,y)\in\Omega\times\Omega$, then for any $\delta<d/2-k$,
\begin{align}\label{eq:outgoing_bound_sol}
    u\mapsto K*u:\Hb^{s-k,\delta+k}(\Omega)\to \Hb^{s,\delta}(\Omega).
\end{align}
Similarly, if $K$ is incoming, i.e. $K_{\rm out}(x,y)=0$ for sufficiently large $(x,y)\in\Omega\times\Omega$, then \eqref{eq:outgoing_bound_sol} holds for any $\delta>-d/2$.
\end{corr}
\begin{proof}
It follows directly from the proof of Proposition \ref{prop:bound_sol}.
\end{proof}

\subsection{Smoothness of curvature}

As a corollary, we know the maps we consider will be smooth.
\begin{corr}\label{cor:bilin_cor}
Let $s>d/2$, $\delta>-d/2$. In a small neighbourhood of $\delta_{ij}$, the inverse matrix map
\begin{align*}
    (g_{ij})\mapsto (g^{ij}): \Hb^{s,\delta}\to \Hb^{s,\delta}
\end{align*}
is a smooth isomorphism.
\end{corr}
\begin{proof}
Since multiplication is bounded by Proposition \ref{prop:b_Sobolev}, it suffices to prove the map
\begin{align*}
    T:h\mapsto \frac{1}{1-h}=1+h+h^2+h^3+\cdots
\end{align*}
is smooth for $\|h\|_{\Hb^{s,\delta}}\ll 1$. The boundedness is a corollary of the bilinear estimate. To prove the boundedness of the derivatives, just observe
\begin{align*}
    DT_g(h)=(1+2g+3g^2+\cdots)h, \quad D^2T_g(h_1,h_2)=(2+3g+4g^2+\cdots)h_1h_2,\cdots
\end{align*}
are all continuous multilinear maps in $\Hb^{s,\delta}$.
\end{proof}
\begin{prop}
For $s>d/2$, $\delta>-d/2$, the functional
\begin{align*}
    (h,\pi)\mapsto(M^{(2)}_h(h,\partial^2 h),M^{(1)}_h(\partial h,\partial h),M^{(0)}_h(\pi,\pi)
    ,N^{(1)j}(h,\partial\pi),N^{(0)j}(\partial h,\pi))
\end{align*}
is smooth $\Hb^{s,\delta}\times \Hb^{s-1,\delta+1}\to \Hb^{s-2,\delta+2}$.
\end{prop}
\begin{proof}
The map is the composition of the inverse matrix map and polynomial maps, which are all smooth by Proposition \ref{prop:b_Sobolev} and Corollary \ref{cor:bilin_cor}.
\end{proof}

\section{Solving the nonlinear equation}\label{4}
In this section we use our solution operators from Theorem \ref{thm:sol} to prove Theorem \ref{thm1} and \ref{thm1’}.
\subsection{Construction of solutions to the linearized equation}
We give a construction of compactly supported solutions of the homogeneous linearized equations \eqref{lin}.
A basic observation is to solve the double divergence equation we only need to solve the symmetric divergence equation
\begin{align}\label{divs}
    \partial_j \pi^{jk}=0
\end{align}
since this would give $\partial_j\partial_k \pi^{jk}=0$.
For \eqref{divs}, we may take an arbitrary $f\in C_0^\infty(\RR^d)$ and
\begin{align*}
    \pi^{11}=\partial_2\partial_2 f,\pi^{12}=-\partial_1\partial_2 f,\pi^{22}=\partial_1\partial_1 f, \pi^{jk}=0 \text{  for  } \{j,k\}\not\subset\{1,2\}.
\end{align*}
Then it is obvious that $\partial_j \pi^{jk}=0$.

\subsection{Fixed point theorem}
Let 
\begin{align*}
    P(h,\pi)&=(\partial_i\partial_j h^{ij},\partial_i\pi^{ij})\\
     \Phi(h,\pi)&=(M^{(2)}_h(h,\partial^2 h)+M^{(1)}_h(\partial h,\partial h)+M^{(0)}_h(\pi,\pi)
    ,N^{(1)j}(h,\partial\pi)+N^{(0)j}(\partial h,\pi)).
\end{align*}
Then the equations \eqref{cons} become
\begin{align*}
    P(h,\pi)=\Phi(h,\pi).
\end{align*}
Let $\Omega$ be a cone centered at $0$ in $\mathbb{R}^d$ and $(h_0,\pi_0)\in C^\infty_0(\Omega)$ be a solution of the linearized equation 
$P(h_0,\pi_0)=0$. Let $S:\Hb^{s-2,\delta+2}(\Omega)\to \Hb^{s,\delta}(\Omega)\times \Hb^{s-1,\delta+1}(\Omega)$ be the solution operator given by Theorem \ref{thm:sol} with $\supp \chi\subset \Omega\cap \mathbb{S}^{d-1}$, we consider the following fixed point problem
\begin{align*}
    (h,\pi)=S\Phi(h_0+h,\pi_0+\pi).
\end{align*}
Since $\Phi:\Hb^{s,\delta}(\Omega)\times \Hb^{s-1,\delta+1}(\Omega)\to \Hb^{s-2,\delta+2}(\Omega)$ is a smooth map with $d\Phi_0=0$, by choosing $\|(h,\pi)\|_{\Hb^{s,\delta}\times \Hb^{s-1,\delta+1}}\leq \epsilon/2$, $\|(h_0,\pi_0)\|_{\Hb^{s,\delta}\times \Hb^{s-1,\delta+1}}\leq \epsilon $ for some sufficiently small $\epsilon>0$ we get
\begin{align*}
    \|\Phi(h_0+h,\pi_0+\pi)\|_{\Hb^{s-2,\delta+2}}&\lesssim \epsilon^2,\\
    \|\Phi(h_0+h,\pi_0+\pi)-\Phi(h_0+\tilde{h},\pi_0+\tilde{\pi})\|_{\Hb^{s-2,\delta+2}}&\lesssim \epsilon\|(h,\pi)-(\tilde{h},\tilde{\pi})\|_{\Hb^{s,\delta}\times \Hb^{s-1,\delta+1}}.
\end{align*}

Since $S$ is again bounded, by Banach fixed point theorem, there exists a unique fixed point $(h_1,\pi_1)\in \Hb^{s,\delta}(\Omega)\times \Hb^{s-1,\delta+1}(\Omega)$ such that $\|(h_1,\pi_1)\|_{\Hb^{s,\delta}\times \Hb^{s-1,\delta+1}}\leq \epsilon/2$ and
\begin{align*}
    (h_1,\pi_1)=S\Phi(h_0+h_1,\pi_0+\pi_1).
\end{align*}
This implies
\begin{align*}
    P(h_0+h_1,\pi_0+\pi_1)=P(h_1,\pi_1)=\Phi(h_0+h_1,\pi_0+\pi_1).
\end{align*}

An alternative way is to notice that 
\begin{align*}
    P-\Phi: \Hb^{s,\delta}(\Omega)\times \Hb^{s-1,\delta+1}(\Omega)\to \Hb^{s-2,\delta+2}(\Omega)
\end{align*}
is a smooth map with surjective differential at $0$ (with a right inverse given by our solution operator). Thus locally $(P-\Phi)^{-1}(0)\cap {\rm nbd}(0)$ is diffeomorphic to $\ker P\cap {\rm nbd}(0)$.

Next we show the solution we get can be actually smooth, even though $s$ is fixed. The regularity comes from applying the solution operator to the nonlinearity to upgrade the regularity.
\begin{prop}\label{reg1}
Let $s>\frac{d}{2}, d\geq 3$. In the above construction, if we choose $(h_0,\pi_0)\in C_0^\infty$ small in $\Hb^{s,\delta}$, then $(h,\pi)\in C^\infty$. 
\end{prop}
\begin{proof}
 We need to come back to the equation. Without loss of generality we may assume $\frac{d}{2}<s<\frac{d+1}{2}$.
Observe that
\begin{align*}
    \partial(h,\pi)=S(M^{(3)}_{h_0,h}(h+h_0,\partial^3 h),N^{(2)j}_{h_0,h}(h+h_0,\partial^2\pi))+\text{controlled terms in } H^{2s-1-d/2}_{\rm loc}\times H^{2s-2-d/2}_{\rm loc}.
\end{align*}
Since $\|h+h_0\|_{\Hb^{s,\delta}}$ is small, we can upgrade the regularity to
\begin{align*}
    (h,\pi)\in  H^{2s-\frac{d}{2}}_{\rm loc}\times H^{2s-1-\frac{d}{2}}_{\rm loc}.
\end{align*}
Keep doing this we will get $( h,\pi)\in H^{d/2+2}_{\rm loc}\times H^{d/2+1}_{\rm loc}$.
Then the equation gives us
\begin{align*}
    \partial(h,\pi)=S(M^{(3)}_{h_0,h}(h+h_0,\partial^3 h),N^{(2)j}_{h_0,h}(h+h_0,\partial^2\pi))+\text{controlled terms in } H^{\frac{d}{2}+2}_{\rm loc}\times H^{\frac{d}{2}+1}_{\rm loc}.
\end{align*}
Since $\|h+h_0\|_{\Hb^{s,\delta}}$ is small, we get $(h,\pi)\in H^{\frac{d}{2}+3}_{\rm loc}\times H^{\frac{d}{2}+2}_{\rm loc}$. Iterating this we have $(h,\pi)\in C^\infty$.
\end{proof}
\subsection{Tail estimate of the solution}
Now we prove the last part of Theorem \ref{thm1}, namely the decay rate estimate \eqref{thm1:decay}. Roughly speaking, if the nonlinearity is integrable, then the decay rate comes from applying the solution operator to the nonlinearity, and should be the same as the decay rate of the solution operator. A priori we do not know whether the nonlinearity is integrable, but we can iterate a few times to get improved estimates.
\begin{prop}
Suppose $d\geq 3, s>\frac{d}{2}, -\frac{d}{2}<\delta<\frac{d}{2}-2$. Let $(g^{ij}-\delta^{ij},k^{ij})\in \Hb^{s,\delta}(\Omega)\times \Hb^{s-1,\delta+1}(\Omega)$ be a solution of equation \eqref{cons} obtained in the previous section, then for $l<s-d-2$, we have
\begin{align}\label{decay1}
    |\partial_l( g^{ij}(x)-\delta^{ij}(x))|\lesssim \langle x\rangle^{2-d-l},\quad |\partial_l k^{ij}(x)|\lesssim \langle x\rangle^{1-d-l}.
\end{align}
\end{prop}
\begin{proof}
We first consider the case $l=0$. By Proposition \ref{prop:b_Sobolev}, we have
\begin{align*}
    |h^{ij}(x)|\lesssim \langle x\rangle^{-(\frac{d}{2}+\delta)}, |\partial^2h^{ij}(x)|\lesssim \langle x\rangle^{-(\frac{d}{2}+\delta+2)}.
\end{align*}
Thus 
\begin{align*}
    |M^{(2)}_{h_0,h}(h_0+h,\partial^2(h_0+h))|\lesssim \langle x\rangle^{-(d+2+2\delta)}.
\end{align*}
The other terms are similar and we have $\Phi(h_0+h,\pi_0+\pi)=\mathcal{O}(\langle x\rangle^{-(d+2+2\delta)})$.
Now 
\begin{align*}
    h=K_\chi*\Phi_1(h_0+h,\pi_0+\pi)
\end{align*}
where the kernel $K=K_\chi$ is homogeneous of degree $2-d$, thus
\begin{align*}
    |h(x)|&=\left|\int K(x-y) \Phi_1(h_0+h,\pi_0+\pi)(y)dy\right|\\
    &\lesssim \int |x-y|^{2-d}\langle y\rangle^{-(d+2+2\delta)}dy\\
    &= \int_{|x-y|<|x|/2}|x-y|^{2-d}\langle y\rangle^{-(d+2+2\delta)}dy+\int_{|x-y|\geq|x|/2}|x-y|^{2-d}\langle y\rangle^{-(d+2+2\delta)}dy.
\end{align*}
If $2+2\delta>0$, then $|h(x)|\lesssim\langle x\rangle^{2-d}$. Otherwise we can only conclude $|h(x)|\lesssim \langle x\rangle^{-(d+2\delta)}$. But we can iterate this process and still conclude
\begin{align*}
    |h(x)|\lesssim \langle x\rangle^{2-d},\quad
    |\pi(x)|\lesssim \langle x\rangle^{1-d}.
\end{align*}
The case $l=1,2$ can be obtained similarly.

Now for general $3\leq l\leq s-d-2$, let us suppose \eqref{decay1} is true for up to $l-1$, and prove for $l$. Recall
\begin{align*}
    |\partial^j h(x)|\lesssim \langle x\rangle^{-(\frac{d}{2}+\delta+j)},
\end{align*}
the induction hypothesis gives
\begin{align*}
    |\partial^l\Phi_1(h_0+h,\pi_0+\pi)|\lesssim \langle x\rangle^{-(3d/2+\delta+l)}.
\end{align*}
Recall $\partial^l h=S\partial^l\Phi_1(h_0+h,\pi_0+\pi)$, so
\begin{align*}
    |\partial^l h(x)|&\lesssim \int_{|x-y|<\frac{|x|}{2}}|x-y|^{2-d}\langle y\rangle^{-(\frac{3d}{2}+\delta+l)}dy+\int_{|x-y|>\frac{|x|}{2}}|x-y|^{2-d-l}\langle y\rangle^{-(3d/2+\delta)}dy\\
    &\lesssim \langle x\rangle^{2-d-l}.
\end{align*}
Similarly we have $|\partial^l\pi(x)|\lesssim \langle x\rangle^{1-d-l}$.
\end{proof}

\subsection{Gluing construction of the solution}
In this section we provide the proof of the gluing result in Theorem \ref{thm1’}.

Let $\Omega\subset \Omega' $ be two cones, after cutting off we only need to solve the constraint equation
\begin{align*}
    P(\chi h_0+h,\chi \pi_0+\pi)=\Phi(\chi h_0+h,\chi \pi_0+\pi)
\end{align*}
inside the region $(\Omega'\cup B_1(y))\setminus (\Omega\cup B_{1/2}(y))$. By choosing $|y|\gg$ 1, we may assume $\|(\chi h_0,\chi \pi_0)\|_{\Hb^{s,\delta}\times \Hb^{s-1,\delta+1}}$ is sufficiently small, and then we can apply the following solution operator to get a solution.
\begin{prop}
Let $\Omega_{\rm int}:=(\Omega'\cup B_1(y))\setminus(\Omega\cup B_{1/2}(y))$, then there is a solution operator
\begin{align*}
    S_{\rm int}:\Hb^{s-2,\delta+2}(\Omega_{\rm int})\to \Hb^{s,\delta}(\Omega_{\rm int})\times \Hb^{s-1,\delta+1}(\Omega_{\rm int}).
\end{align*}
\end{prop}
\begin{proof}
Let $f\in \Hb^{s-2,\delta+2}(\Omega_{\rm int})$. First we can move the support outside $B_2(y)$ using the explicit Bogovskii-type solution operator $S_0$ constructed in \cite{maoohtao2022}. Let $\chi_0\in C_0^\infty(B_2(y))$, then
\begin{itemize}
    \item $PS_0(\chi_0f)=\chi_0f$ inside $B_2(y)$;
    \item $\supp S_0(\chi_0f)\subset B_3(y)$;
    \item $S_0(\chi_0f)\in \Hb^{s,\delta}(\Omega_{\rm int})\times \Hb^{s-1,\delta+1}(\Omega_{\rm int})$.
\end{itemize}

Now we make a partition of unity in angular variables and use the solution operator on each piece. Suppose $\Omega_{\rm int}\cap \mathbb{S}^{d-1}_y=\cup U_i$ and each $U_i$ is star-shaped with respect to an open subset $V_i\subset U_i$. Now let $\chi_i\in C_0^\infty(V_i)$ be a cutoff function supported in $V_i$, then as in Theorem \ref{thm:sol} we have solution operators $S_{\chi_i}$ with respect to $U_i$ so that $\supp u\subset y+\RR_{>1}(U_i-y)$ implies $\supp S_{\chi_i}u\subset y+\RR_{>1}(U_i-y)$. We take a partition of unity $\tilde{\chi}_j$ with respect to the covering $\{U_i\}$ and define
\begin{align*}
    S_1=\sum\limits_j S_{\chi_j} \tilde{\chi}_j.
\end{align*}
The final solution operator $S_{\rm int}$ is defined to be
\begin{align*}
    S_{\rm int}f=S_1(f-PS_0(\chi_0f)) + S_0(\chi_0f)\in \Hb^{s,\delta}(\Omega_{\rm int})\times \Hb^{s-1,\delta+1}(\Omega_{\rm int}).
\end{align*}
One can check it is bounded $H^{s-2}_{\delta+2}(\Omega_{\rm int})\to H^s_\delta(\Omega_{\rm int})\times H^{s-1}_{\delta+1}(\Omega_{\rm int})$ and
\begin{equation*}
     PS_{\rm int}f=PS_1(f-PS_0(\chi_0f)) + PS_0(\chi_0f)=f-PS_0(\chi_0f) + PS_0(\chi_0f)=f.\qedhere
\end{equation*}

\end{proof}

\begin{rem}
    The norm of $\Hb^{s,\delta}(\Omega_{\rm int})$ should be defined with respect to the center $y$, to make uniform estimates in $y$.
\end{rem}
Following the procedure of the previous section, we can get a gluing solution from the solution operator $S_{\rm int}$. The proof for the smoothness and decay rate is identical to the previous argument.

\section{Solving the problem in a degenerate sector}\label{sec:Omega}
In this section we want to find solutions of the Einstein constraint equations in the degenerate sector $\{(x',x_d)\in\mathbb{R}^d\,:\,  |x'|\leq x_d^\alpha\}$ for some $\alpha<1$. For technical convenience, we define
\begin{align*}
    \Omega=\{(x',x_d)\in\mathbb{R}^d\,:\,  |x'|\leq x_d^\alpha:x_d\geq 1\}
\end{align*}
and indeed construct the solution in $\Omega$.

In order to analyze the regularity of the solution, we need an anisotropic weighted Sobolev space defined as follows. For simplicity, we will only consider integer $s\in\NN$ in this section.
\begin{defi}
Let $\alpha\in (0,1]$, $s\in\mathbb{N}$ be a non-negative integer, for $u\in C_0^\infty(\Omega)$,
\begin{align*}
    \|u\|_{H^{s,\delta}_{\alpha}}^2:=\sum\limits_{|\beta|\leq s}  \|\langle x\rangle^{|\beta'|\alpha+\beta_d+\delta} \partial_{x'}^{\beta'}\partial_{x_d}^{\beta_d} u\|_{L^2}^2.
\end{align*}
\end{defi}

We provide a few properties of this norm. 
\begin{prop}
\begin{itemize}
Let $u\in C_0^\infty(\Omega)$, then
    \item For $s>d/2$, $\|\langle x\rangle^{\frac{(d-1)\alpha+1}{2}+\delta}u\|_{L^\infty}\lesssim \|u\|_{H^{s,\delta}_{\alpha}}$. More generally, for $s>d/2-d/p$, $\|\langle x\rangle^{((d-1)\alpha+1)(\frac{1}{2}-\frac{1}{p})+\delta}u\|_{L^p}\lesssim \|u\|_{H^{s,\delta}_{\alpha}}$.
    \item For $s>d/2$, $(g,h)\mapsto gh$ is is continuous $H^{s,\delta_1}_{\alpha}\times H^{s,\delta_2}_{\alpha}\to H^{s,\delta_1+
    \delta_2+\frac{(d-1)\alpha+1}{2}}_{\alpha}$
\end{itemize}
\end{prop}
\begin{proof}
The first estimate follows from rescaling. The bilinear estimate is also direct by choosing $p_1,p_2$ according to $\beta$ for H\"older inequality and use the first Sobolev embedding estimate:
\begin{align*}
    &\|gh\|_{H^{s,\delta_1+\delta_2+\frac{(d-1)\alpha+1}{2}}_{\alpha}}^2=\sum\limits_{|\beta|\leq s}  \|\langle x\rangle^{|\beta'|\alpha+\beta_d+\delta_1+\delta_2+\frac{(d-1)\alpha+1}{2}} \partial_{x'}^{\beta'}\partial_{x_d}^{\beta_d} (gh)\|_{L^2}^2\\
    &\lesssim \sum\limits_{|\beta|+|\gamma|\leq s} \|\langle x\rangle^{|\beta'|\alpha+\beta_d+\delta_1+\frac{(d-1)\alpha+1}{p_2}} \partial_{x'}^{\beta'}\partial_{x_d}^{\beta_d} g\|_{L^{p_1}}^2\|\langle x\rangle^{|\gamma'|\alpha+\gamma_d+\delta_2+\frac{(d-1)\alpha+1}{p_1}} \partial_{x'}^{\gamma'}\partial_{x_d}^{\gamma_d} h\|_{L^{p_2}}^2\\
    &\lesssim \|g\|_{H^{s,\delta_1}_{\alpha}}^2\|h\|_{H^{s,\delta_2}_{\alpha}}^2.
\end{align*}
In the second last step we choose $1/p_1+1/p_2=1/2$ such that $s-|\beta|>d/2-d/p_1$ and $s-|\gamma|>d/2-d/p_2$ (except for the end point case $|\gamma|=s$ or $|\beta|=s$ which is clear from Sobolev embedding).
\end{proof}

Next, we turn to the construction of a solution operator for the double divergence equation $\partial_i\partial_j h^{ij}=f$. We start with the divergence equation $\partial_j v^j=f$ as before.
The key is finding fundamental solutions of the divergence equation supported on a half-curve.
\begin{prop}
For any smooth curve $\gamma:[0,\infty)\to \RR^d$ with $\gamma(0)=y$ and $\lim\limits_{t\to\infty}\gamma(t)=\infty$, the distribution $\delta_\gamma$ defined as
\begin{align*}
    (\delta_\gamma,\varphi)=\int_0^\infty \varphi(\gamma(t))\gamma'(t)dt
\end{align*}
is a fundamental solution for the divergence equation
\begin{align*}
    \partial_j u^j=\delta_y.
\end{align*}
\end{prop}
\begin{proof}
Let $\varphi\in C_0^\infty(\RR^d)$, then
\begin{equation*}
    (\partial_j\delta_{\gamma}^j,\varphi)=-\int_0^\infty \partial_j\varphi(\gamma(t))\gamma'^j(t)dt=-\int_0^\infty \partial_t \varphi(\gamma(t)) dt=\varphi(\gamma(0))=\varphi(y).\qedhere
\end{equation*}
\end{proof}

As before, we average a family of fundamental solutions to get a more regular solution supported in $\Omega$.
\begin{lem}\label{lem:deg_div} Suppose $\delta<\frac{\alpha(d-1)-1}{2}$.
    We have a solution operator $\Tilde{S}_0:H^{s-1,\delta+1}_{\alpha}(\Omega)\to H^{s,\delta}_{\alpha}(\Omega)$ for the divergence equation, i.e.
    $\partial_j \tilde{S}^j_0 f=f.$
\end{lem}
\begin{proof}
We construct the solution operator in two steps.

Let $\gamma^{(1)}_{y,\omega}=y+(\omega y_d^\alpha,y_d)t$, $\chi_1\in C^\infty_0(\mathbb{R}^{d-1})$, and define
\begin{align*}
    K_1=\int_{\mathbb{R}^{d-1}}\chi_1(\omega)\delta_{\gamma^{(1)}_{y,\omega}}d\omega=\chi_1\left(\frac{(x'-y')/y_d^\alpha}{(x_d-y_d)/y_d}\right)\frac{(x-y)}{y_d^{\alpha(d-1)+1}|\frac{x_d-y_d}{y_d}|^d}.
\end{align*}
Then $\Div K_1=\delta(x-y)$. Let $\chi_2$ be a cutoff and 
\begin{align*}
    \tilde{K}_1=\chi_2\left(\frac{x_d-y_d}{y_d}\right)K_1.
\end{align*}
Then $\Div\tilde{K}_1 f = f + \displaystyle\frac{1}{y_d}\chi_2'\left(\frac{x_d-y_d}{y_d}\right)K_1$. Let $\gamma^{(2)}_{y,\omega}(t)= (y'+\omega ((1+t)^\alpha-1),y_d+ t)$ and 
\begin{align*}
    K_2&=\int_{\mathbb{S}^{d-1}}\chi_1(\omega)\delta_{\gamma^{(2)}_{y,\omega}}d\omega\\
    &=\chi_1\left(\frac{x'-y'}{(1+x_d-y_d)^\alpha-1}\right)((1+x_d-y_d)^\alpha-1)^{-(d-1)}\left(\alpha(x'-y')\frac{(1+x_d-y_d)^{\alpha-1}}{(1+x_d-y_d)^\alpha-1},1\right).
\end{align*}
Then $K(x,y)=\tilde{K}_1(x,y)-\displaystyle\int K_2(x,z)\frac{1}{y_d}\chi_2'\left(\frac{z_d-y_d}{y_d}\right)K_1(z,y)dz$ is a solution operator of the divergence equation. Let $\tilde{K}_2(x,y)=-\displaystyle\int K_2(x,z)\frac{1}{y_d}\chi_2'\left(\frac{z_d-y_d}{y_d}\right)K_1(z,y)dz$, it is straightforward to verify
\begin{itemize}
    \item $K$ is outgoing;
    \item $\partial_{x'}^{\beta'}\partial_{x_d}^{\beta_d} \tilde{K}_2(x,y)\lesssim |x_d-y_d|^{-\alpha(d-1)-|\beta'|\alpha-\beta_d}$;
\end{itemize}
We need to estimate the solution operator in both regions.
\begin{itemize}
    \item In the nearby region (i.e. $x\sim y$), we would like to show
    \begin{align}\label{degsol_bound_near}
        \left\|\int \tilde{K}_1(x,y) u(y)dy\right\|_{H^{s,\delta}_{\alpha}}\lesssim \|u\|_{H^{s-1,\delta+1}_{\alpha}}.
    \end{align}
    It suffices to prove it in a fixed annulus $A_R:=\{R<|y_d|<2R\}$. We claim by rescaling, we may assume $R=1$ without loss of generality.
    Let $\tilde{K}_1(x,y)=(K_{s}(x,y), K_b(x,y))$, we observe that
    \begin{align*}
        \tilde{K}_1(R^\alpha x',Rx_d, R^\alpha y', R y_d)=R^{-\alpha(d-1)-1}(R^\alpha K_s(x,y),R K_b(x,y)),
    \end{align*}
    so given the estimate \eqref{degsol_bound_near} for $R=1$ we get
    \begin{align*}
        &\left\|\int \tilde{K}_1(x,y) u(y)dy\right\|_{H^{s,\delta}_{\alpha}(A_R)}^2\\
        &\lesssim  R^{\alpha(d-1)+1}R^{2\delta}\left\|R^{\alpha(d-1)+1}\int \tilde{K}_1(R^\alpha x', Rx_d,R^\alpha y', R y_d) u(R^\alpha y', Ry_d)dy\right\|_{H^s(A_1)}^2\\
        &\lesssim R^{\alpha(d-1)+1}R^{2+2\delta}\left\|\int  \tilde{K}_1(x,y) u(R^\alpha y', Ry_d)dy\right\|_{H^s(A_1)}^2\\
        &\lesssim R^{\alpha(d-1)+1}R^{2+2\delta}\left\|u(R^\alpha y', Ry_d)\right\|_{H^{s-1}(A_1)}^2\\
        &\lesssim \|u\|_{H^{s-1,\delta+1}_{\alpha}(A_R)}^2.
    \end{align*}
    On the unit annulus, we notice $\tilde{K}_1(x,y)$ gives a pseudodifferential operator of order $-1$. If we let $b(y_d,z)=\displaystyle\chi_1\left(\frac{z'/y_d^\alpha}{z_d/y_d}\right)\frac{(z'/y_d^\alpha, z_d/y_d)}{y_d^{\alpha(d-1)+1}|\frac{z_d}{y_d}|^d}$ and $a(\xi)=\mathcal{F}(b(1,\cdot))$, then $a$ is homogeneous of degree $-1$ and
    \begin{align*}
        (y_d^{-\alpha}K_s(x,y), y_d^{-1}K_b(x,y))= \frac{1}{(2\pi)^d}\chi_2\left(\frac{x_d-y_d}{y_d}\right)\int e^{i(x-y)\cdot\xi} a(y_d^\alpha\xi',y_d\xi_d)d\xi.
    \end{align*}
     Let $\chi$ be a cutoff near $\xi=0$, then $(1-\chi(\xi))a(y_d^\alpha \xi',y_d\xi_d)$ is a symbol in the sense that
     \begin{align*}
         |\partial_y^\beta\partial_\xi^\gamma (1-\chi(\xi))a(y_d^\alpha \xi',y_d\xi_d)|\lesssim_{\beta,\gamma} \langle \xi\rangle^{-1-\gamma|\xi|}.
     \end{align*}
     Thus the pseudodifferential operator maps from $H^{s-1}$ to $H^s$. Moreover, the rest part $\int  e^{i(x-y)\cdot\xi}\chi(\xi)a(y_d^\alpha\xi',y_d\xi_d)d\xi\in C^\infty$ has smooth Schwartz kernel, thus gives a smoothing operator. So we conclude \eqref{degsol_bound_near} from the discussion above.
    
\item In the faraway region, we can estimate directly.
\begin{align*}
    \|\tilde{K}_2u\|_{H^{s,\delta}_{\alpha}(\Omega)}^2&=\sum\limits_{|\beta|\leq s}  \|\langle x\rangle^{|\beta'|\alpha+\beta_d+\delta} \partial_{x'}^{\beta'}\partial_{x_d}^{\beta_d} (\tilde{K}_2u)\|_{L^2}^2\\
    &\lesssim \sum\limits_{|\beta|\leq s}\sum\limits_{j'}2^{2j'(|\beta'|\alpha+\beta_d+\delta-\alpha(d-1)-|\beta'|\alpha-\beta_d)+((d-1)\alpha+1)j'} \left(\sum\limits_{j'>j} 2^{((d-1)\alpha+1)j/2}\| \psi_j u\|_{L^2}\right)^2\\
    &\lesssim \sum\limits_{j'}2^{j'(2\delta-(d-1)\alpha+1)}\sum\limits_{j<j'}2^{-2j(\delta+1)}2^{\gamma(j-j')}\sum\limits_{j<j'}2^{\gamma(j'-j)}2^{2j(\delta+1)}2^{((d-1)\alpha+1)j}\|\psi_j u\|_{L^2}^2\\
    &\lesssim \sum\limits_j 2^{2j(\delta+1)}\|\psi_j u\|_{L^2}^2 \\
    &\lesssim \|\langle x\rangle^{\delta+1} u\|_{L^2}^2
\end{align*}
where we take $2\delta+2<\gamma<(d-1)\alpha+1$ which is possible since $\delta<\frac{(d-1)\alpha-1}{2}$.\qedhere
\end{itemize}

\end{proof}
From the solution operator for the divergence equation, we also get the solution operator for the double divergence equation.
\begin{prop}
Suppose $\delta<\frac{\alpha(d-1)-3}{2}$, then there is a solution operator $\tilde{S}:H^{s-2,\delta+2}_{\alpha}(\Omega)\to H^{s,\delta}_{\alpha}(\Omega)$ for the double divergence equation, i.e. $\partial_i\partial_j \tilde{S}^{ij} f=f$ and $\tilde{S}^{ij}$ is symmetric. Moreover, the integration kernel $K(x,y)$ of $\tilde{S}$ satisfies 
\begin{align*}
     |\partial_{x'}^{\beta'}\partial_{x_d}^{\beta_d} K(x,y)|\lesssim \langle x\rangle^{1-\alpha(d-1)-|\beta'|\alpha-\beta_d},\quad x_d>2y_d.
\end{align*}
\end{prop}
\begin{proof}
We just need to apply $\tilde{S}_0$ twice and symmetrize it:
\begin{align*}
    \tilde{S}^{ij}f=\frac{1}{2}( \tilde{S}_0^i\tilde{S}_0^j f+\tilde{S}_0^j\tilde{S}_0^i f).
\end{align*}
The bound of the tail follows from the construction in Lemma \ref{lem:deg_div}.
\end{proof}

For the symmetric divergence equation, the construction is trickier. We use methods from \cite{reshetnyak1970estimates} and refer to \cite{div_eq2023} for further discussions.

\begin{prop}
    Suppose $\delta<\frac{\alpha(d+1)-3}{2}$, then there exists a solution operator $\Tilde{L}:H^{s-1,\delta+2-\alpha}_\alpha(\Omega)\to H^{s,\delta}_\alpha(\Omega)$ for the symmetric divergence equation, i.e. $\partial_i\tilde{L}^{ij}_k f_j= f_k$ and $\tilde{L}^{ij}$ is symmetric. Moreover, the integration kernel $K(x,y)$ of $\Tilde{L}$ satisfies 
    \begin{align*}
        |\partial_{x'}^{\beta'}\partial_{x_d}^{\beta_d} K(x,y)|\lesssim \langle x\rangle^{1-\alpha d-|\beta'|\alpha-\beta_d},\quad x_d>2y_d.
    \end{align*}
\end{prop}
\begin{proof}
Fix a smooth curve $\gamma(t):[0,\infty)\to\RR^d$ such that $\gamma(0)=y$ and $\lim\limits_{t\to \infty}\gamma(t)=\infty$. We want to find $L\in\mathcal{D}'(\RR^d)$ such that 
\begin{align*}
    \varphi_k(y)=\langle \partial_{i} L^{ij}_{k},\varphi_j\rangle=-\frac{1}{2}\langle L^{ij}_k,\partial_i\varphi_j+\partial_j\varphi_i\rangle,\quad \varphi_j\in C_0^\infty(\RR^d).
\end{align*}
In order to recover $\varphi_k(y)$ from $\zeta_{ij}=-\frac{1}{2}(\partial_i\varphi_j+\partial_j\varphi_i)$, we let $\eta_{ij}=\frac{1}{2}(\partial_i\varphi_j-\partial_j\varphi_i)$ so that
\begin{align}\label{eq:div_eq_1}
    \partial_i\varphi_j=\eta_{ij}-\zeta_{ij}.
\end{align}
Since 
\begin{align*}
    \partial_{j}\zeta_{ik}=-\frac{1}{2}\partial_{ij}^2\varphi_k-\frac{1}{2}\partial_{jk}^2\varphi_i, \quad \partial_k\zeta_{ij}=-\frac{1}{2}\partial_{ik}^2\varphi_j-\frac{1}{2}\partial_{jk}^2\varphi_i,
\end{align*}
we have
\begin{align}\label{eq:div_eq_2}
    \partial_i\eta_{jk}=\frac{1}{2}(\partial^2_{ij}\varphi_k-\partial^2_{ik}\varphi_j)=\partial_k\zeta_{ij}-\partial_j\zeta_{ik}.
\end{align}
Now we can integrate \eqref{eq:div_eq_2} along $\gamma$ and get
\begin{align*}
    \eta_{jk}(\gamma(t))=-\int_t^\infty (\gamma'(s)^i\partial_k\zeta_{ij}(\gamma(s))-\gamma'(s)^i\partial_j\zeta_{ik}(\gamma(s))  )ds.
\end{align*}
Then we can integrate \eqref{eq:div_eq_1},
\begin{align*}
    \varphi_j(y)&=-\int_0^\infty ( \gamma'(t)^i\eta_{ij}(\gamma(t))-\gamma'(t)^i\zeta_{ij}(\gamma(t)))dt\\
    &=\int_0^\infty (\gamma'(s)^l\partial_j\zeta_{li}(\gamma(s))-\gamma'(s)^l\partial_i\zeta_{lj}(\gamma(s))  ) \left(\int_0^s \gamma'(t)^idt\right) ds +\int_0^\infty\gamma'(t)^i\zeta_{ij}(\gamma(t))dt\\
    &=\int_0^\infty (\gamma'(t)^l\partial_j\zeta_{li}(\gamma(t))-\gamma'(t)^l\partial_i\zeta_{lj}(\gamma(t))  )  (\gamma(t)^i-\gamma(0)^i) dt +\int_0^\infty\gamma'(t)^i\zeta_{ij}(\gamma(t))dt.
\end{align*}
So the fundamental solution $L^{i}$ supported on the curve $\gamma$ is given by
\begin{align*}
    \langle L^{ij}_k,\zeta_{ij}\rangle=\int_0^\infty (\gamma'(t)^j\partial_k\zeta_{ij}(\gamma(t))-\gamma'(t)^j\partial_i\zeta_{jk}(\gamma(t))  )  (\gamma(t)^i-\gamma(0)^i) dt +\int_0^\infty\gamma'(t)^i\zeta_{ik}(\gamma(t))dt.
\end{align*}
We then average along curves as in Lemma \ref{lem:deg_div}. For $\gamma^{(1)}_{y,\omega}=y+(\omega y_d^\alpha,y_d)t$, $\omega\in\RR^{d-1}$, $\chi_1\in C^\infty_0(\mathbb{R}^{d-1})$, we have
\begin{align*}
    L_1:&=\int_{\RR^{d-1}} \chi_1(\omega) L_{y,\omega}^{(1)}d\omega=-\partial_k\left(\chi_1\left(\frac{(x'-y')/y_d^\alpha}{(x_d-y_d)/y_d}\right)\frac{(x_j-y_j)(x_i-y_i)}{y_d^{\alpha(d-1)+1}|x_d-y_d|^d/y_d^d}\right)\\
    &+\frac{1}{2}\partial_l\left(\chi_1\left(\frac{(x'-y')/y_d^\alpha}{(x_d-y_d)/y_d}\right)\frac{(x_l-y_l)((x_j-y_j)\delta_{ik}+(x_i-y_i)\delta_{jk})}{y_d^{\alpha(d-1)+1}|x_d-y_d|^d/y_d^d}\right)\\
    &+\frac{1}{2}\chi_1\left(\frac{(x'-y')/y_d^\alpha}{(x_d-y_d)/y_d}\right)\frac{(x_i-y_i)\delta_{jk}+(x_j-y_j)\delta_{ik}}{y_d^{\alpha(d-1)+1}|x_d-y_d|^d/y_d^d}.
\end{align*}
For $\gamma_{y,\omega}^{(2)}(t)=(y'+\omega((1+t)^\alpha-1),y_d+t)$, we have
\begin{align*}
    L_2:&=\int_{\RR^{d-1}} \chi_1(\omega) L_{y,\omega}^{(2)}d\omega\\
    &=-\frac{1}{2}\partial_k\left(\chi_1\left(\frac{x'-y'}{(1+x_d-y_d)^\alpha-1}\right)((1+x_d-y_d)^\alpha-1)^{-(d-1)}\right.\\&\left(\alpha(x'-y') \frac{(1+x_d-y_d)^{\alpha-1}}{(1+x_d-y_d)^\alpha-1},1 \right)^j(x_i-y_i)+\left(\alpha(x'-y') \frac{(1+x_d-y_d)^{\alpha-1}}{(1+x_d-y_d)^\alpha-1},1 \right)^i(x_j-y_j)\Biggl)\\&
    +\frac{1}{2}\partial_l\Biggl(\chi_1\left(\frac{x'-y'}{(1+x_d-y_d)^\alpha-1}\right)(x_l-y_l)((1+x_d-y_d)^\alpha-1)^{-(d-1)}\\
    &\left(\left(\alpha(x'-y') \frac{(1+x_d-y_d)^{\alpha-1}}{(1+x_d-y_d)^\alpha-1},1 \right)^j\delta_{ik}+\left(\alpha(x'-y') \frac{(1+x_d-y_d)^{\alpha-1}}{(1+x_d-y_d)^\alpha-1},1 \right)^i\delta_{jk}\right)\Biggl)\\
    &+\frac{1}{2}\chi_1\left(\frac{x'-y'}{(1+x_d-y_d)^\alpha-1}\right)((1+x_d-y_d)^\alpha-1)^{-(d-1)}\\
    &\left(\left(\alpha(x'-y') \frac{(1+x_d-y_d)^{\alpha-1}}{(1+x_d-y_d)^\alpha-1},1 \right)^i\delta_{jk}+\left(\alpha(x'-y') \frac{(1+x_d-y_d)^{\alpha-1}}{(1+x_d-y_d)^\alpha-1},1 \right)^j\delta_{ik}\right).
\end{align*}
We then define the solution operator $\tilde{L}$ from $L_1$ and $L_2$ as in Lemma \ref{lem:deg_div} and it follows from the same proof that $\tilde{L}:H^{s-1,\delta+2-\alpha}_{\alpha}(\Omega)\to H^{s,\delta}_{\alpha}(\Omega)$
and
\begin{equation*}
    |\partial_{x'}^{\beta'}\partial_{x_d}^{\beta_d}K(x,y)|\lesssim \langle x\rangle^{1-\alpha -\alpha(d-1)-|\beta'|\alpha-\beta_d},\quad x_d>2y_d.\qedhere
\end{equation*}
\end{proof}

We can now give the proof of Theorem \ref{thm2}:
\begin{proof}[Proof of Theorem \ref{thm2}]
Let $d\geq 3, s>\frac{d}{2}+2$ be an integer. Let  $\frac{3}{d+1}<\alpha<1$ and $\frac{3-(d+3)\alpha}{2}<\delta<\frac{\alpha(d-1)-3}{2}$.
We now choose small but nontrivial $C^\infty_0$ solutions $(h_0,\pi_0)$ of the linearized equation $P(h_0,\pi_0)=0$, and solve the fixed point problem 
\begin{align*}
    (h,\pi)=(\tilde{S},\tilde{L})\Phi(h_0+h,\pi_0+\pi).
\end{align*} 
on the space $H^{s,\delta}_{\alpha}(\Omega)\times H^{s-1,\delta+\alpha}_\alpha(\Omega)$. Note $\delta<\frac{\alpha(d-1)-3}{2}$ ensures the solution operators $\tilde{S},\tilde{L}$ map to the correct spaces, and the other condition $\delta>\frac{3-(d+3)\alpha}{2}$ ensures the bilinear estimate 
$$\|\Phi(h_0+h,\pi_0+\pi)\|_{H^{s-2,\delta+2}_\alpha}\lesssim C(\|h_0+h\|_{H^{s,\delta}_\alpha})\|h_0+h\|_{H^{s,\delta}_\alpha}\|\pi_0+\pi\|_{H^{s-1,\delta+\alpha}_\alpha}.$$
By Banach fixed point theorem, we get a unique solution 
\begin{align*}
    (g^{ij}-\delta^{ij},k^{ij})\in H^{s,\delta}_{\alpha}(\Omega)\times H^{s-1,\delta+\alpha}_\alpha(\Omega).
\end{align*}

The smoothness and decay rate for $(h,\pi)$ is proved similarly as before.
\end{proof}

\printbibliography

\end{document}